\documentclass[11pt]{article}
\usepackage{amssymb,latexsym,amsfonts,verbatim,amscd}
\usepackage[all]{xy}

\newtheorem{Def}{Definition}[section]

\newtheorem{Lem}[Def]{Lemma}
\newtheorem{Prop}[Def]{Proposition}

\newtheorem{Rem}[Def]{Remarks}

\font\nat msbm10 scaled\magstephalf
\def\N{\hbox{\nat\char78}}

\def\telos{\hfill$\dashv$}

\begin{document}

\title{An observer-based approach to the sorites paradox and the logic derived from that}
\author{Athanassios Tzouvaras}

\date{}
\maketitle

\begin{center}
Aristotle University of Thessaloniki \\
Department  of Mathematics  \\
541 24 Thessaloniki, Greece \\
  e-mail:\verb"tzouvara@math.auth.gr"
\end{center}

\begin{abstract}
We approach the sorites paradox (SP) through an observer-based and time-dependent approach to truth of vague assertions. Formally the approach gives rise to a semantics,  called fluxing-object semantics (FOS),  because it  involves  models  that contain ``fluxing objects'', that is, entities changing with time and observer. The models are equipped with agents (observers) and a linear and discrete time axis for time. The changing entities  are represented by partial functions of time and agent, and this partiality causes truth-value gaps. If we interpret  a truth-value gap as a third truth value, then  FOS becomes a three-valued logic, which, quite interestingly, is proved identical to strong Kleene three-valued logic.   The sorites phenomena can be represented in a structure of FOS as special objects that change imperceptibly with respect to an observer and with respect to  a particular attribute. In  this account  the  key point that eliminates the paradoxical character of sorites is the partiality of functions. When the observer is fixed the partiality  corresponds  to interrupted watching on his part. The interruption creates watching gaps during which the attributed property  of the object as understood by the observer  may change, without violating the imperceptibility condition. Interrupted watching has, according to experts on visual  attention and focusing capabilities of humans, a  firm physiological justification. The relationship of watching gaps with horizon crossing is also discussed.
\end{abstract}

\vskip 0.1in

{\em Mathematics Subject Classification (2020)}:  03B80, 06B50

\vskip 0.1in

{\em Key words and phrases.} Sorites paradox, vague predicate,  watching gap, fluxing-object semantics, imperceptibly changing object, strong Kleene logic.

\section{Introduction}

The sorites paradox (henceforth abbreviated SP) is a well-known and  persistent puzzle arising in the  area of vagueness, specifically in  discrete vagueness, i.e., vague predicates which refer to quantities measured by natural numbers (baldness, bigness, heapness, etc).

Mathematically  SP amounts to failure of induction: Some vague predicates create  subsets of $\N$ without least element.  For example if $P(n)$ denotes the property  ``$n$ is small'' (or ``$n$ is not big''),  $n$ a positive integer, then clearly the extensions of both $P(n)$ and $\neg P(n)$ are nonempty. Moreover, since common wisdom necessitates $P(n)\Rightarrow P(n+1)$, the transition from the extension of $P(n)$
to that of $\neg P(n)$  is  done in a mysterious and
incomprehensible way, without crossing any boundary and without
one being able to tell  when and  how this transition is
realized. We call such a transition from $P(n)$ to $\neg P(n)$, or
conversely,  a {\em  vague transition.} Vague transitions are of
course genuinely paradoxical (i.e., contradictory) for the set
$\N$ of ordinary integers.  If we insist to assign set extensions
to such predicates, then the only mathematical
option available seems to be the use of nonstandard models of
arithmetic, which do permit the existence of sets without least
element.
This approach has been  taken in \cite{Tz98}, where it is
shown that such models do represent faithfully most of the
fundamental aspects of vague predicates. For example if $M$ is a
nonstandard model with $\N\subset M$, then the sets $\N$ and
$M\backslash \N$ model a vague transition, since $\N$  and $M\backslash \N$ have no
last and first elements, respectively.
But nonstandard integers are totally infeasible and hence
unable to be used in everyday experience (more in section 4.2).

In this paper we  provide a  formal description of vague
transitions without leaving the ground of standard natural numbers. As for the logic that the changing (in fact fluxing) entities obey,  we shall adopt a special semantics suitable for their description, that  heavily relies on  a factor not taken into account in
standard treatments of SP (e.g. those of \cite{Fi75}, \cite{Ke00} and \cite{Sh06}). Namely, the role of an {\em observer} who is present at each particular
sorites event and announces whether (they believe that) a given object   $a$ satisfies $P$ or non-$P$. This semantics, called fluxing object semantics (FOS),  will be considered in  section 2.

\subsection{The content  of the paper}
Our approach adopted  here to  SP is based to  an agent-dependent and time-dependent formalization of truth. It is a variant of classical semantics of  propositional  logic, with the difference that the objects of the ground model are not constant but ``fluxing'', i.e., constantly changing, as {\em partial} functions of time and observer.  It is called fluxing-object semantics (FOS). In FOS we have a time axis represented by $\N$ and a set $I$ of observers (although we mostly  refer to a single observer $i$). The changing objects are represented by partial (actually finite) functions from $I\times \N$ into a classical relational  structure ${\cal A}=\langle A,R_1,R_2,\ldots\rangle$. The partiality of  functions $f$ makes it possible to have points $t<t'<s$ on the time axis such that $f(i,t)$ and $f(i,s)$ are defined, while $f(i,t')$ is not. This is interpreted as a ``watching gap'' between $t$ and $s$ concerning the observation of the object $f$ by the observer $i$. Watching gaps entail that the observation of objects and situations is generally done interruptedly, or intermittently.  The total period of watching splits into  disjoint  intervals  of time, with watching gaps among them. By the definition of imperceptibly changing situations  no vague transition is possible  during the time of
uninterrupted perceptual focusing.  Therefore if any vague transition takes place within a specific changing situation,  it takes place  only during a watching gap. Since there  can be plenty of watching gaps in the formal representation of objects, this easily provides  a solution to SP.

Another  remarkable feature of  FOS is the following. As expected, the partiality of the functions used in the FOS domains, i.e. the watching gaps, produces truth-value gaps in its satisfaction relation $\Vdash$.  If  we interpret these gaps as a third truth value (undetermined), then FOS gives rise to a three-valued logic. It is   interesting that  this logic (in the form of truth tables for the connectives $\wedge$, $\vee$ and $\neg$) is identical to the strong Kleene three-valued logic.

The paper is organized as follows: In  section 2  we  sketch the framework of FOS and prove that it is a kind of sound and incomplete logic. In section 3 we treat the paradox SP in the framework of FOS and show how it is eliminated.   In section 4 we  cite the views of some experts in the fields of cognitive psychology and theory of vision on the question of plausibility of the idea of watching gaps. We also discuss how watching gaps explain horizon crossings. Finally in the Appendix we prove that FOS has the same truth tables with strong Kleene logic.

\section{Fluxing Object Semantics (FOS)}
In contrast to the objects of mathematics which are abstract and
timeless, physical objects are always time dependent and  most
often also agent dependent. So, whereas  17 has been and is going
to be  a prime number for ever and for any person who understands
the definition of a prime, a man $a$ can be nonbald at a time $t$
for an agent $i$ and bald at a time $t'>t$ for the same agent.
Also $a$ can be nonbald for an agent $i$ and bald for another
agent $j$ at the same time $t$. In the aforementioned  examples
two kinds of things are involved: Individuals (either abstract
like 17, or concrete like a specific man) and properties
(primeness, baldness). In fact  in the Fregean-Russellian
tradition the entire world is made of individuals and properties
of individuals. This tradition provides  a rather
simplified picture of the world but  we shall follow it here because it
underlies the whole of first-order logic. Classical logic is the
most secure ground for our  thinking, but it's rather a narrow one. For instance  it fails to treat (directly) changing entities. Our aim
here is to modify the standard  semantics in order to
accommodate truths about such entities which vary with time and  agent.

Let us start with  a first-order relational language
$L=\langle P_1,P_2,\ldots\rangle$,  where each $P_k$ is a predicate symbol of
some arity $n_k>0$. $L$ might contain also constants, but below we
shall use names for all elements of the structures interpreting
$L$, so special constants would be redundant.  Formulas and
sentences of $L$ are defined as usual.  A first-order $L$-structure is as usual a tuple ${\cal A}=\langle A,R_1,R_2\ldots\rangle$ where
$A$ is a nonempty set and each $R_k$ is an $n_k$-relation on $A$
that interprets $P_k$, i.e., $P_k^{\cal A}=R_k$. The elements of
$A$ are supposed to represent physical objects, like quantities
of sand, or snapshots of a gradually and slowly
changing entity, like the hairy part of one's head, or the stages
of a ripening fruit, etc. Given ${\cal A}$, $L(A)$ denotes the
language $L\cup A$, i.e., $L$ augmented with the elements of $A$
as parameters (names of themselves). Then for every sentence
$\phi$ of $L(A)$ ${\cal A}\models \phi$ is defined as usual.

In order now to define  objects over ${\cal A}$ which vary with
respect to time and agent, first we fix a nonempty set of {\em
agents} $I$.  Letters  $i,j$ range over elements of $I$. Time
will be represented by an  infinite, linearly and discretely ordered set $\langle T,<\rangle$. For simplicity throughout our set of moments will be identified with $\langle \N,<\rangle$. However we shall keep writing $t$ for time moments.
And of course  $t+1$ denotes the moment next to $t$.

A varying  object over ${\cal A}$ will be a {\em finite}
(hence partial) function $f:I\times \N\hookrightarrow A$
(the notation $f:X\hookrightarrow Y$ means that $dom(f)$ is a  subset of $X$).
The set $A$ is construed here as the set  of ``object forms'' of
the varying objects, satisfying certain conditions discussed
below. Intuitively, if $f(i,t)$ is defined, then
\begin{equation} \label{E:intended}
f(i,t) \ \mbox{:=``the form of $f$ perceived by $i$
at time $t$''}.
\end{equation}
If $rng(f)$ happens to be a singleton $\{a\}$, then we may think of $f$ as a constant  entity, not dependent on time and observer,
e.g. like the abstract entities of mathematics.
But if e.g.  $f$ is a specific man, say John, and $i$ is an observer, say Peter, and if Peter  {\em happens to watch} John
at time $t$, then  $f(i,t)$ denotes  John's {\em instant picture} or {\em form}\footnote{The contrast between the object-function $f$ and
its forms-values $f(i,t)$ roughly reminds the Platonic dichotomy
of beings into  ideas and their forms. $f$ is the lasting ``idea''
or ``identity'' of the object behind its temporal forms $f(i,t)$. Objects flow
but something underlies this fluxing. In contrast,  the older
Heraclitean model of everything being just a  fluxing form, denies that
anything stable underlies the flux. That would formally look like
having just values, not coming out of a specific function.}
perceived by Peter at $t$. Otherwise $f(i,t)$ is not defined. Note that $f$ is finite if and only if $dom(f)$ is finite. In practice, since we always focus on the behavior of a single agent $i$, we are interested  in the properties of the functions $f_i$,  where
$$f_i(t):=f(i,t),$$
rather than $f$.  The reason that we require $f$ to be  finite rather than just  partial, is that   if $I$ has more than one element, then $f$ may be partial but for some $i\in I$, it could be   $dom(f_i)=\N$. That would mean that $i$ watches the object $f$ constantly along the entire time axis, which is strongly unrealistic. Even if we require that for each $i$, $f_i$ is partial, still $dom(f_i)$ might be infinite. The finiteness condition for $f$ makes things more natural, as  being in accord with the finite resources of real observers.

What about relations? Do they also vary with respect to time or
agent or both? The natural answer is no. According to common wisdom  properties are abstract and stable with respect to time and agents, at least much more stable than the physical entities. So if for example  Peter  says that
``John  is nonbald'' at a time $t$, and at a later time $t'>t$  he changes his view to  ``John  is bald'', we do not attribute this
change to a change of the predicate of baldness, but rather to a
change of John. Thus we shall assume throughout that, in contrast
to objects,  predicates $P_1,P_2,\ldots$ are {\em independent} of time
and agent.\footnote{On the other hand there exist predicates which indeed depend on the context where they are used. For example the predicate ``big'' has quite different  meaning and  extension when it is used to describe   the number of stars of the galaxy,  the number of  grains of sand in a heap, and the number of pupils in a school class. So throughout we assume that each predicate $P_i$ is always interpreted in some fixed  context. }

As usual we write
$f(i,t)\!\!\downarrow$, or $f_i(t)\!\!\downarrow$,  if  $\langle i,t\rangle\in dom(f)$, and $f(i,t)\!\!\uparrow$ otherwise. In view of the intended meaning of $f(i,t)$ given in (\ref{E:intended}), two coexisting objects $f,g$ should satisfy the following condition:

\vskip 0.1in

(\dag) \quad \quad If  $f\neq g$ then  $rng(f)\cap rng(g)=\emptyset$.

\vskip 0.1in

\noindent This is because $rng(f)$ is the ``trajectory'' of $f$ in the world, and no
two distinct objects can have overlapping trajectories.
On the other hand if  $f\neq g$ we may have  $dom(f)\cap dom(g)\neq \emptyset$, which means that  an observer $i$ can watch at time $t$  simultaneously two or more different objects. For instance this is the case of an observer watching both John and Bill and announcing   ``John is taller than Bill''. In general an observer is supposed to be able to watch a situation which consists of many distinct constituents.

\begin{Def} \label{D:special}
{\em Let ${\cal A}=\langle A,R_1,R_2,\ldots\rangle$ be a first-order structure for
the language $L=\langle P_1,P_2,\ldots\rangle$. Let $I$ be a set of agents, and
let $Fin(A^{I\times \N})$ denote the set of all finite functions
with $dom(f)\subset I\times \N$ and $rng(f)\subset A$. A } fluxing
structure over ${\cal A}$ {\em is  a structure ${\cal
F}=\langle F,R_1,R_2,\ldots\rangle$ such that  $F\subseteq Fin(A^{I\times \N})$ and the elements of $F$ satisfy condition (\dag) above. }
\end{Def}
Several characteristics of the objects $f$ can be defined. What we shall need below is only  the following. Recall that for every $f\in F$ and $i\in I$, $f_i$ is the function defined by $f_i(t)=f(i,t)$. Let
$$W_i(f)=dom(f_i)=\{t:\langle i,t\rangle\in dom(f)\}.$$
$W_i(f)$  is the set of all time moments at which $i$ watches the object $f$.

Given a fluxing structure  ${\cal F}=\langle F,R_1,R_2\ldots\rangle$, let $L(F)$
denote as usual  $L$ augmented with the elements of the
set $F$ as parameters.

What we need  below is actually a notion of ``situated truth'' just for atomic sentences of the form  predicates
$P_k(f_1,\ldots,f_{n_k})$, and their negations, where $P_k$ is an $n_k$-relation symbol
and $f_1,\ldots,f_{n_k}\in F$, with respect to ${\cal F}$ and a
pair $\langle i,t\rangle\in I\times\N$, denoted $${\cal F},
\langle i,t\rangle\Vdash P_k(f_1,\ldots, f_{n_k}).$$ However, for reasons of
completeness we give a definition of ${\cal F}, \langle i,t\rangle\Vdash\phi$
for every propositional sentence $\phi$ of $L(F)$, i.e. a sentence built from  atomic sentences $P_k(f_1,\ldots, f_{n_k})$, for $f_i\in F$,  with  connectives $\wedge$, $\vee$ and $\neg$.

\begin{Def} \label{D:support}
{\em Given a fluxing structure ${\cal F}=\langle F,R_1,R_2,\ldots\rangle$ over
${\cal A}$, a pair $\langle i,t\rangle\in I\times \N$ and a sentence $\phi$ of
$L(F)$, the relation ``${\cal F},i,t\Vdash\phi$'' is defined
inductively as follows: :

a) ${\cal F},i,t\Vdash P_k(f_1,\ldots, f_{n_k})$ iff
$f_1(i,t)\!\downarrow \ \& \cdots  \& \ f_{n_k}(i,t)\!\downarrow$  $
\& \ {\cal A}\models P_k(f_1(i,t),\ldots, f_{n_k}(i,t))$, i.e., if all $f_{n_l}(i,t)$ are defined and  $\langle f_1(i,t),\ldots, f_{n_k}(i,t)\rangle\in R_k$.

b) ${\cal F},i,t\Vdash\neg P_k(f_1,\ldots, f_{n_k})$ iff
$f_1(i,t)\!\downarrow \& \cdots \ \& \ f_{n_k}(i,t)\!\downarrow$  $ \&
\ {\cal A}\models \neg P_k(f_1(i,t),\ldots, f_{n_k}(i,t))$, i.e., if all $f_{n_l}(i,t)$ are defined and $\langle f_1(i,t),\ldots, f_{n_k}(i,t)\rangle\notin R_k$.

c) ${\cal F},i,t\Vdash\phi \wedge \psi$ iff ${\cal F},i,t\Vdash\phi$ and ${\cal F},i,t\Vdash\psi$.

d) ${\cal F},i,t\Vdash\phi \vee \psi$ iff ${\cal F},i,t\Vdash\phi$ or ${\cal F},i,t\Vdash\psi$.

e) ${\cal F},i,t\Vdash\neg(\phi \wedge \psi)$ iff ${\cal F},i,t\Vdash\neg \phi \vee \neg \psi$.

f) ${\cal F},i,t\Vdash\neg(\phi \vee \psi)$ iff ${\cal F},i,t\Vdash\neg \phi \wedge \neg \psi$.

g)  ${\cal F},i,t\Vdash\neg\neg\phi$ iff  ${\cal F},i,t\Vdash\phi$. }
\end{Def}

We shall refer to the semantics defined above as {\em fluxing-object semantics} (FOS).

\vskip 0.1in

FOS is a kind of logic without axioms and rules of inference, but only  a semantic part which is given through the satisfaction relation $\Vdash$. We  can   read ${\cal F},i,t\Vdash\phi(f)$ as {\em ``${\cal F}$ supports $\phi$ at
$\langle i,t\rangle$'',} or, better, {\em ``$i$ sees (or believes, or thinks)
that  $f$ satisfies $\phi$ at time $t$ in ${\cal F}$''}. We show below that FOS is sound, i.e., no structure proves contradictions, but, due to the partiality of the functions in every domain ${\cal F}$,  it is incomplete, that is there are $\phi$ such that for some $i,t$, ${\cal F},i,t\!\not\Vdash\phi$ and ${\cal F},i,t\!\not\Vdash\neg\phi$.  This is why in Def. \ref{D:support} we need to define also the truth of the  negations of complex formulas.

\begin{Lem} \label{L:sound} (i) FOS is sound: there are no ${\cal F}$, $i$, $t$ and $\phi$ such that  ${\cal F},i,t\!\Vdash\phi$ and ${\cal F},i,t\!\Vdash\neg\phi$.

(ii) FOS is incomplete: there are ${\cal F}$, $i$, $t$ and $\phi$ such that ${\cal F},i,t\!\not\Vdash\phi$ and ${\cal F},i,t\!\not\Vdash\neg\phi$.
\end{Lem}

{\em Proof.} (i) By induction on the complexity of $\phi$. (1) If $\phi$ is atomic, it is of the form $P_k(f_1,\ldots,f_n)$, for some predicate $P_k$ of $L$ and  some  $f_1,\ldots,f_n\in F$. Then the claim follows  from clauses (a) and (b) of Def. \ref{D:support}. (2) If the claim holds for $\phi$, it holds also for $\neg\phi$ because of clause (g) of \ref{D:support}.  (3) We assume the claim holds for $\phi$ and $\psi$ and prove it for $\phi\wedge\psi$. Suppose ${\cal F},i,t\!\Vdash\phi\wedge\psi$ and ${\cal F},i,t\!\Vdash\neg(\phi\wedge\psi)$. By clause (e) of \ref{D:support} the latter is equivalent to ${\cal F},i,t\!\Vdash\neg\phi\vee\neg\psi$. So we have both $\phi$ and   $\psi$ to be true and also some of the $\neg\phi$ and $\neg\psi$ to be true, which contradicts the induction assumption. (4) Similarly we show the claim for $\phi\vee\psi$ in view of clause (f) of \ref{D:support}.

(ii)  Take for example $\phi(x)$ to be the unary predicate $P(x)$ of baldness.  Let also $f\in F$ be a fluxing object, $i$ be an observer and  $t\in \N$ such  that  $f(i,t)$ is not defined, i.e. $f(i,t)\!\uparrow$. Then by clauses (a), (b) of  \ref{D:support}, ${\cal F},i,t\!\not\Vdash P(f)$ and ${\cal F},i,t\!\not\Vdash\neg P(f)$. \telos

\vskip 0.2in

Now if we interpret ``${\cal F},i,t\!\Vdash\phi$'' as ``$\phi$ is true in ${\cal F},i,t$''  and ``${\cal F},i,t\Vdash\neg\phi$'' as ``$\phi$ is false  in ${\cal F},i,t$'', and write  ${\cal F}_{i,t}(\phi)=\textsf{T}$ and  ${\cal F}_{i,t}(\phi)=\textsf{F}$, respectively, then we can interpret ``${\cal F},i,t\!\not\Vdash \phi$ and ${\cal F},i,t\!\not\Vdash\neg \phi$'' as ``$\phi$ is undetermined in ${\cal F},i,t$'' and write  ${\cal F}_{i,t}(\phi)=\textsf{U}$. This way FOS can be seen as a three-valued logic, with truth values $\{\textsf{T},\textsf{U},\textsf{F}\}$ and designated value $\textsf{T}$. An interesting fact about this logic is that  coincides with the strong Kleene three-valued logic, in the sense that their truth tables for $\wedge$, $\vee$ and $\neg$ are identical. This result is proved in the Appendix at the end of the paper.

\vskip 0.2in

By definition the objects of ${\cal F}$ are changing. Of particular
importance, however, are the objects that change {\em imperceptibly}
with respect to the agents they watch them. (For simplicity we assume that imperceptibility is an inter-subjective condition, i.e., if an object changes imperceptibly with respect to some  agent $i$, it does so with respect to any other agent $j$.)

\begin{Def} \label{D:ic}
{\em Let ${\cal F}=\langle F,R_1,R_2,\ldots\rangle$ be fixed, let $f\in
F\subseteq  Fin(A^{I\times \N})$ and let $\phi(x)$ be a property
of $L(F)$. $f$ is said to be } imperceptibly changing (i.c.) with
respect to $\phi$ {\em if
$$(\forall i\in I)(\forall t\in \N)[f(i,t)\!\downarrow \& f(i,t+1)\!\downarrow \Rightarrow (\langle i,t\rangle\Vdash\phi(f) \Leftrightarrow \langle i,t+1\rangle\Vdash\phi(f))]. $$}
\end{Def}

Intuitively, $f$ is i.c.\footnote{In using the term
``imperceptibly changing'', the emphasis is on ``imperceptibly''
rather than on ``changing''. That is to say, $f$ is i.c. even if
it doesn't change at all, i.e., if $f(i,t)=f(i,t+1)$. } with
respect to $\phi$, if for every agent $i$, if he/she  watches $f$ at two consecutive moments $t$ and $t+1$, then he believes that  $f$ satisfies
$\phi$ at  $t$ if and only if  he believes that $f$ satisfies $\phi$ at $t+1$. The reason of course for the existence of i.c. objects with
respect to a property $\phi$, is that the rate of change, which is related to the predicate $\phi$, is very low w.r.t. to the human
discerning abilities. For example aging presumably causes a constant
change in human face and body but this change cannot be detected within some   days or months or even years.

\vskip 0.2in

{\bf Examples of i.c. objects}

1) Let $M$ be a specific man, and let $A$ be a set which, possibly
among other things, contains the snapshots of the hairy part of
$M$'s head, taken by several observers $i\in I$ at several time
moments $t\in \N$ along a period of time. If  $f\in Fin(A^{I\times
\N})$ and $f(i,t)\!\downarrow$, we assume that  $f(i,t)$ is the
snapshot taken by $i$ at time $t$. Let $P(x)$ be the unary
predicate of baldness in the language $L$ which is interpreted by
the unary relation $R\subseteq A$ of baldness over $A$. So we are
interested in the first-order structure ${\cal A}=\langle A,R\rangle$ and a
fluxing structure ${\cal F}$ over ${\cal A}$. Actually every $f\in
Fin(A^{I\times \N})$ is an i.c. object with respect to $P$ since
baldness is a property not changing,
with respect to any observer, within two consecutive moments. That
is, for any two consecutive snapshots $f(i,t)$ and $f(i,t+1)$ by
the same observer $i$, if they are both defined, then either
${\cal A}\models P(f(i,t)) \wedge P(f(i,t+1))$ or ${\cal A}\models
\neg P(f(i,t)) \wedge \neg P(f(i,t+1))$. So if $f(i,t)\!\downarrow$
and $ f(i,t+1)\!\downarrow$, then  $\langle i,t\rangle\Vdash P(f) \iff
\langle i,t+1\rangle\Vdash P(f)$. Therefore $f$ satisfies the conditions of
definition \ref{D:ic}.

\vskip 0.1in

2) As another example, let $D$ be a device pouring out a grain of
sand per unit of time and let $A$ be the set of growing
quantities of sand that constitute snapshots taken by observers  at
various time moments.  For every  $f\in Fin(A^{I\times \N})$,
$f(i,t)$ denotes  the quantity of sand perceived by $i$ at $t$ (if
$i$ does watch the sand at $t$). Let $P(x)$ denote the predicate
of ``heapness'' and let $H\subseteq A$ be the extension of this property
on $A$.  That is, if $a\in A$ is a quantity of sand, $a\in H$
means that  $a$ is a heap. We are interested in the structure
${\cal A}=\langle A,H\rangle$ and a fluxing structure ${\cal F}$ over ${\cal
A}$. Since a single grain of sand cannot affect, according to
common wisdom, the property of heapness, it is clear  that every $f\in Fin(A^{I\times \N})$ is an
i.c. object with respect to the predicate $P$.\footnote{Note that in each case  the length of time units must be chosen according to the specific i.c. situation we are going to describe. If e.g. the sand is poured out  very rapidly, while the time unit used by the observer is considerably great, then most
likely the condition of definition \ref{D:ic} will be violated, since
an observer $i$ may decide that $f(i,t)$ is not a heap but
$f(i,t+1)$ {\em is}. For example if the observer's time unit is  one
second, while the rate of change of the heap is one grain of sand per  millisecond, then   1000 elementary changes will be  accumulated at each observer's snapshot,  which of course may be summed up into a perceptible change. So $n$ should always count time periods very close to the rate at which  the situation is changing.}

\vskip 0.1in

3) The third example is somewhat more ``abstract''. Consider  a
mechanical number counter displaying   successive integers
$1,2,\ldots, 1013,1014,\ldots$ in increasing order. Suppose that at
time $n$ the number $k+n$ is being displayed (where $k$ is a
constant depending on the counter). $A$ is the set of displayed
integers and $P(x)$ is a predicate for largeness. An agent $i$
watches the monitor from time to time and claims that the
displayed integer is large or non-large. If $f(i,t)\!\downarrow$,
then $f(i,t)$ is the integer observed by $i$ at $t$ and
$P(f(i,t))$ or $\neg P(f(i,t))$ is his claim, respectively. Let
$R\subseteq A$ interprets $P$. We are considering the structure
${\cal A}=\langle A,R\rangle$ and a fluxing structure ${\cal F}$ over ${\cal
A}$. $\langle i,t\rangle\Vdash P(f)$ expresses the fact that ``$i$ observes
the integer $f(i,t)$ at $t$  and believes that $f(i,t)$ is
large''. Since largeness is not affected by adding or subtracting
1, clearly every $f\in Fin(A^{I\times \N})$ is an i.c. object with respect to the predicate $P$.

\vskip 0.1in

$\textbf{Remark}$. Concerning the first of the preceding examples, which refers to the hairy part of the head of a man $M$,  one may object that taking a snapshot of $M$'s head,  an observer possibly could not decide whether $M$ is bald or not bald but, using modifying adverbs, he/she may say that the man is ``almost bald'', or ``completely bald'', or ``definitely not bald'' and the like, something which is quite common for vague notions. This is true but it does affect  the {\em core} of the sorites paradox itself. The core of SP is the {\em vague transition} from a quality $P$ to a {\em different} quality $Q$ ``without crossing any boundary and without one being able to tell  when and  how this transition is realized'', as we said in the Introduction. For the sake of emphasis, we take these different qualities to be  opposite to one another, $P$ and $\neg P$, or  ``bald'' and ``non-bald''. It would make no difference if instead we took these qualities to be ``bald'' and ``almost bald''. For in that case  the question would accordingly be how a man passes {\em vaguely}  by the time  from the state of ``almost baldness'' to the state of ``baldness''.

\section{SP in the framework of  FOS}
Let's see how SP is treated in FOS. Recall
from definition \ref{D:special} that to every object $f\in
Fin(A^{I\times \N})$ and every agent $i$, there corresponds the
watching time of $f$ by $i$ $$W_i(f)=dom(f_i)=\{t\in \N:\langle i,t\rangle\in
dom(f)\},$$ which is, by definition, a  finite subset of $\N$. For every
$t\leq s\in \N$, let $[t,s]$ denote the closed interval $\{n\in
\N:t\leq n\leq s\}$. Trivial one-element intervals $[n]=\{n\}$ are allowed. Now clearly  every
finite set $X\subset \N$ has a unique analysis into  a finite
union of disjoint finite  consecutive and maximal  intervals of $X$
$$X=\bigcup_{m=1}^p[t_m,s_m].$$ Due to consecutiveness and maximality,  for each $m$ $[t_m,s_m]<[t_{m+1},s_{m+1}]$ and  $s_m+1<t_{m+1}$. For otherwise for some $m$, $s_m+1=t_{m+1}$. But then $[t_m,s_{m+1}]$ is an interval of $X$  and  $[t_m,s_m]\subset
[t_m,s_{m+1}]\subseteq X$, which contradicts the maximality of $[t_m,s_m]$. Therefore for every $m=1,\ldots,p-1$,
\begin{equation} \label{E:conditions}
t_m\leq s_m \ \mbox{and} \ \ s_m+1<t_{m+1}.
\end{equation}
We call this partition of $X$ into disjoint maximal intervals satisfying conditions (\ref{E:conditions})
the {\em canonical analysis} of $X$. In particular, for every
$f,i$, $W_i(f)$, as a finite subset of $\N$, has a canonical analysis
\begin{equation} \label{E:maximal}
W_i(f)=\bigcup_{m=1}^p[t_m,s_m].
\end{equation}
$W_i(f)$ may consist either of a single interval of time, or of
several maximal disjoint intervals. The maximal intervals $[t_m,s_m]$ that constitute $W_i(f)$ are the {\em constituents} of $W_i(f)$.

\begin{Def} \label{D:gap}
{\em Let $t,s\in W_i(f)$ such that $t<s$. We say  that there is a} watching gap between $t$ and $s$ (for the agent $i$ with respect to the object $f$) {\em if there is a $t'\in \N$ such that $t<t'<s$ and $f_i$ is not defined at $t'$, i.e., $t'\notin dom(f_i)$. }
\end{Def}

\textbf{Note.} After the completion of this article, I came across   P. Pagin's work \cite{Pa10} and \cite{Pa11}, which bears some remarkable similarities  to this paper. Namely, the role played  in the present framework by the notion of  ``watching gap'' is played in Pagin's work by his notion of  ``central gap.'' Actually the two notions are surprisingly similar. The difference lies rather in the philosophical motivation behind them as well as in their implementation in different contexts.

\vskip 0.1in

A watching gap may be either a single time moment or an interval of $\N$. Equivalently, there exists a watching gap between $t<s$, if and only if  $W_i(f)$ consists of more than one constituent and $t,s$ do not belong to the same constituent. Obviously if $t<t'<s$ and $f(i,t')!\uparrow$, then $t'$ is located  between two consecutive constituents of $W_i(f)$, of the form $[t_k,s_k]$ and $[t_{k+1}, s_{k+1}]$ (see Figure 1.)

$$\overbrace{\cdots[_{t_m}\mbox{---$\cdot_{t}$---}\cdot ]_{s_m}\cdots[_{t_k}\mbox{----}]_{s_k}\underbrace{\quad  \cdot_{t'} \quad }_{\mbox{\scriptsize watching gap}}[_{t_{k+1}}\mbox{----}]_{s_{k+1}}\cdots[_{t_l}
\mbox{---$\cdot_{s}$---}]_{s_l}\cdots}^{W_i(f)}$$
\begin{center}
Figure 1
\end{center}

Existence of watching gaps for an agent $i$ who watches an object $f$, obviously implies {\em interrupted watching.} That is, while $i$ focuses on $f$ at a time $t$, at a time $t'>t$ he is distracted from $f$, turning to another object, say $g$,  and at a time $s>t'$ he focuses back to  $f$. Such a behavior is quite common  and natural concerning human observers as we shall argue in the last section.

\begin{Def} \label{D:shift}
{\em Let $i,f$ be as above, let $\phi(x)$ be a property of  $L$
with respect to which $f$ is an i.c. object, and let  $t,s\in
W_i(f)$ such that $t<s$. We say that} $i$  changes his view with respect to
$\phi$ between $t$ and $s$ {\em if $\langle i,t\rangle\Vdash\phi(f)$ and
$\langle i,s\rangle\Vdash\neg \phi(f)$, or $\langle i,t\rangle\Vdash\neg \phi(f)$ and $\langle i,s\rangle\Vdash\phi(f)$. }
\end{Def}

\begin{Prop} \label{P:onlyif}
Let $f$ be an $i.c.$ object with respect to a property $\phi$.  If an agent $i$ changes his view  with respect to $\phi$
between the moments  $t$ and  $s$, for $t<s$, then necessarily  there is a watching gap between $t$ and $s$.
\end{Prop}

{\em Proof.} This is a rather immediate consequence of the
definition of i.c. objects with respect to $\phi$. It is clear
from definition \ref{D:ic} that if $t,t+1\in W_i(f)$, then $\langle i,t\rangle\Vdash
\phi(f) \Leftrightarrow \langle i,t+1\rangle\Vdash \phi(f)$. Let  $t<s\in W_i(f)$ and suppose there  is no watching gap between $t$ and $s$. Then clearly $[t,s]\subseteq W_i(f)$. If $s-t=n$, then by induction we
see that for every $0\leq k\leq n$ $\langle i,t\rangle\Vdash \phi(f)
\Leftrightarrow \langle i,t+k\rangle\Vdash \phi(f)$. Hence for $k=n$,
$\langle i,t\rangle\Vdash \phi(f) \Leftrightarrow \langle i,s\rangle\Vdash \phi(f)$, which
means that there is no change of view with respect to $\phi$
between  $t$ and  $s$, a contradiction. \telos

\vskip 0.2in

Proposition \ref{P:onlyif} says that whenever a change of view occurs with respect to an i.c. object, this is always due to the  existence of watching gaps. Since watching gaps are available in abundance by the very definition of the fluxing objects $f\in F$ and the structure of $W_i(f)$, Proposition \ref{P:onlyif} offers by the same token a solution to SP: Whenever a change of view occurs on the part of the agent $i$, just conclude that a watching gap for $i$ has  occurred between two contradictory assessments of the i.c. object $f$. This way SP ceases to be  a paradox.

\begin{Rem} \label{R:rec}
{\em (1) $W_i(f)$ need not always  contain watching gaps. For some $i$ and $f$ $W_i(f)$  may  well consist  of a unique interval of $\N$. In such a case however, for every i.c. property $\phi$ concerning $f$, $i$ cannot  change his view with respect to $\phi$ along $W_i(f)$.

(2) If   $t,s\in W_i(f)$, $t<s$ and there is a watching gap between $t$ and $s$, $i$ need {\em not}  change his view about $\phi$ along the interval $[t,s]$. That is,  the change of view is a {\em possible} rather than  a necessary phenomenon when a watching  gap occurs.}
\end{Rem}

\section{Some discussion}
In this last  section we discuss two themes: 1) If there is experimental evidence that real observers observe interruptedly, and 2) how the idea of watching gap can explain the rather mysterious crossing of a horizon.

\subsection{Do actual observers watch interruptedly?}
The question  enters the area of cognitive psychology and attention focusing theory, so  experts in these fields could provide  dependable assessments of it, possibly on the grounds  of
experimental evidence. I contacted  some specialists in
the field of cognitive psychology and vision. One of them is Erick
Weichselgartner who, together with G. Sperling, have supported in
\cite{SW95} the view that attention is  discrete (i.e.,
``quantal'', or ``episodic) rather than continuous. When I asked
Weichselgartner about the plausibility of the existence of watching gaps he said (personal
communication):
\begin{quote}
``Interrupted watching is well supported by empirical psychological findings.
Whatever the task: Subjects' error rates, and/or reaction times
increase by the time.''
\end{quote}

He said also  that the keywords in our case is not just
``attention'' but rather {\em vigilance} or {\em alertness} or
{\em sustained attention.} To my question whether ``humans (and
perhaps all mammals) observe, think and in general communicate
discontinuously and intermittently, in an on and off way, all the
time'', he replied:

\begin{quote}
``There is a controversy in the psychological literature as to
whether attention is continuous or discrete (see \cite{SN96},
\cite{SC82}, \cite{SW95}). As far as I can tell this question is
still unanswered. From a psychologist's point of view a negative
attitude towards interrupted watching  is in disagreement with the research
literature. I do not quite understand  your mathematical work, but
your assumptions with respect to psychology are fine!''
\end{quote}

I contacted also John Tsotsos, Distinguished Professor of Vision Science at York University. He  wrote to me (personal communication):
\begin{quote}
``When humans look at something, they cannot do so for a long period of time simply because of neural fatigue. If you fixate on a point with your eyes and do not move your eyes or blink or anything, the neurotransmitter material in the photoreceptors in your retina gets used up and has no time to be replenished. So things begin to fade - details in particular. Moving the eyes distributes the neurotransmitter material] where events fall on the retina and thus different receptors are used at different times. Think of a screen saver  that uses different portions of your screen at different times. It helps preserve the screen. Same thing in your eyes. And this actually applies to all the visual neurons in the visual cortex.''
\end{quote}

I find the above remark particularly important for the plausibility and justification of interrupted watching, since it explains in terms of the physiology of vision why a real observer cannot stare at something for a long period of time.

For various side effects due to the phenomenon of neural fatigue, like the fading of sharp black and white contours into a more muted, or gray, field, the adaptation of motion,  scintillation and shimmering, etc, Professor Tsotsos referred to \cite{So94} (p. 60).

\subsection{Interrupted watching  and  horizon crossing}

Let us first give  a brief overview of the notion of ``horizon''  which goes back to Husserl.  P. Vop\v{enka} uses also extensively this notion in his book \cite{Vo79} on the alternative set theory. A horizon  with respect to an observer $i$,  is a sequence of steps (of any sort: physical, perceptual, intellectual) stretching ahead  of $i$ with no  last element   {\em visible} by $i$. As such   it is a highly relative and subjective idea, yet quite indispensable in everyone's grasping of the reality. For example,  the collection of time moments  of one's life at which he thinks he is young is a horizon. There is also an ultimate horizon for every human, his own life-time itself. Every conscious agent knows that the world  extends beyond any particular horizon. Namely,  given any horizon $h$ there are steps   within $h$ and steps  beyond $h$. We refer to the transition from within to beyond $h$ as the {\em horizon crossing.} Obviously the latter coincides with what we called in the introduction ``vague transition'', i.e., the imperceptible transition from a property $P$ to its opposite $\neg P$.

Let us consider again the typical example of  the collection $Small_i$ of  numbers which are perceived by $i$ as small.  It gives  rise to  a horizon, since at each particular  time $t$, $i$ cannot identify a last element of $Small_i$, i.e., $n\in S_i$ implies $n+1\in Small_i$.  Abstracting from $t$, i.e., replacing  a real observer by an idealized one, the  absence of last element makes every such sequence look  potentially infinite.  Concerning the example just cited, this means that $Small_i$ looks infinite and yet $\N-Small_i\neq \emptyset$. A representation of this situation would require  a partition of $\N$ like the one  depicted in Figure 2, that is, a splitting of  $\N$  into an initial segment with no last element (represented by the dots) and a final segment with no  first  element.

$$\overbrace{\underbrace{\mbox{---------------------------}\cdots}_{Small_i})
(\underbrace{\cdots\mbox{----------------------}}_{non-Small_i}}^{\N}$$
\begin{center}
Figure 2
\end{center}
Of course such a splitting of the actual time axis is impossible, and this is in  fact equivalent to the sorites paradox.  The situation is theoretically remedied if one passes to a proper  extension of $\N$, i.e., to a nonstandard model $M$ of natural numbers. Such models exist in abundance  and are studied extensively in model theory. For any such $M$,  $\N$ is a proper initial segment of $M$, and moreover  $M\backslash \N$ does not contain a first element. Thus replacing $\N$ by  $M$, $Small_i$ by $\N$ and $not$-$Small_i$ by $M\backslash \N$, we have the picture of Figure 3 which is structurally identical  to that of Figure 2, but also mathematically possible.
$$\overbrace{\underbrace{\mbox{---------------------------}\cdots}_{\N})
(\underbrace{\cdots\mbox{----------------------}}_{M\backslash \N}}^{M}$$
\begin{center}
Figure 3
\end{center}
In \cite{Tz98} I have proposed such an implementation by means of sets of nonstandard  natural numbers. Such sets do capture nicely many theoretical aspects of vagueness, but the main disadvantage is the  infeasibility of nonstandard numbers, i.e. the elements of $M\backslash \N$, since they  are absolutely remote from our real world, and so no practical understanding of real vagueness phenomena can provide.  The horizons of real life occur necessarily within $\N$, and it's only with respect to $\N$ that their paradoxes should  be raised. Connections of non-standard analysis with the sorites paradox can be found also in \cite{De18} and \cite{Itz21}.

Now interrupted watching   offers indeed a realistic and consistent  implementation  of  horizon crossing   within the existent time axis $\N$. It says that horizon crossings  occur when we do not pay attention to them, when we ``are not there'', i.e.,  when  we do not watch. So instead of the situation of Figure 2, which concerns idealized observers, in real life we have the situation shown in Figure 4. For as long as we consciously attend  the growth of  small numbers, they remain small. At some time we shall quit and turn our attention to something  else. If after a period of  time  we turn  our attention back to the growing numbers again, we may find that they have already become large.

$$\overbrace{\underbrace{\mbox{----------------------}}_{Small_i}]\underbrace{ \quad \quad \quad \quad}_{\mbox{\scriptsize watching gap}} [\underbrace{\mbox{-----------------}}_{non-Small_i}}^{\N}$$
\begin{center}
Figure 4
\end{center}

So to the question  ``what is the practical tool by which horizons can be  implemented in real world,'' I believe  the answer is: the implementation of  horizons  through interrupted watching, as described above, is not only satisfactory  but rather the only possible one. One can give numerous
examples of horizon crossings as manifestations of interrupted watching. Each
time we meet a friend after a sufficiently long time, and watch
the changes in his/her face, the wrinkles that he hadn't before, the
hairs that grew slightly grey etc.,  we know that this person has
crossed a horizon. He/she has {\em perceivably} changed in the
period between the two encounters, during which he/she has been
unwatched by us. On the contrary, we do
not realize the gradual changes of our spouses with whom we are
together all the time. We cannot witness a  thing
at the {\em very} moment  of horizon crossing simply because this moment belongs to a watching gap. So this moment is going to be  irreparably elusive.  To put it more  poetically,  our lives consist of  alternating  encounters (watchings) and separations (interruptions), contacts and avoidances. Horizons are crossed only in between.

\vskip 0.3in

\begin{center}
APPENDIX
\end{center}

\begin{center}
\textbf{The three-valued logic of FOS}
\end{center}

Recall that as follows from  Definition \ref{D:support} and the subsequent discussion,  there exist $\phi$ and ${\cal F},i,t$ such that ${\cal F},i,t\not\Vdash \phi$ and ${\cal F},i,t\not\Vdash \neg\phi$. If we interpret ``${\cal F},i,t\!\Vdash\phi$'' as ``$\phi$ is true in ${\cal F},i,t$''  and ``${\cal F},i,t\Vdash\neg\phi$'' as ``$\phi$ is false  in ${\cal F},i,t$'', and write ${\cal F}_{i,t}(\phi)=\textsf{T}$ and  ${\cal F}_{i,t}(\phi)=\textsf{F}$, respectively, then we can interpret ``${\cal F},i,t\!\not\Vdash \phi$ and ${\cal F},i,t\!\not\Vdash\neg \phi$'' as ``$\phi$ is undefined in ${\cal F},i,t$'' and write ${\cal F}_{i,t}(\phi)=\textsf{U}$. Specifically we set:

\begin{Def} \label{D:truthvalues}
We define for every sentence $\phi$:

$\bullet$ ${\cal F}_{i,t}(\phi)=\textsf{T}$ iff ${\cal F},i,t\Vdash \phi$.

$\bullet$ ${\cal F}_{i,t}(\phi)=\textsf{F}$ iff ${\cal F},i,t\Vdash \neg\phi$.

$\bullet$ ${\cal F}_{i,t}(\phi)=\textsf{U}$ iff ${\cal F},i,t\not\Vdash \phi$ and ${\cal F},i,t\not\Vdash \neg\phi$.
\end{Def}

Notice that by Lemma \ref{L:sound} (i), i.e. the soundness of FOS, the above definition is unambiguous, that is, there are  no ${\cal F}$, $i$, $t$, and $\phi$ such that  ${\cal F}_{i,t}(\phi)$ is  at the same $\textsf{T}$ and $\textsf{F}$. So FOS is a sound three-valued logic, with truth values $\{\textsf{T},\textsf{U},\textsf{F}\}$ and designated value $\textsf{T}$.
We shall prove below that  this logic is {\em identical} to strong Kleene logic, in the sense that their truth tables for $\neg$, $\wedge$ and $\vee$ are identical. In the next two lemmas for simplicity we drop the subscript from the operator ${\cal F}_{i,t}(\cdot)$ and write just  ${\cal F}(\cdot)$.

\begin{Lem} \label{L:neg}
For every $\phi$:

(i) ${\cal F}(\phi)=\textsf{F} \ \Leftrightarrow {\cal F}(\neg\phi)=\textsf{T}$.

(ii) ${\cal F}(\phi)=\textsf{T} \ \Leftrightarrow {\cal F}(\neg\phi)=\textsf{F}$.

(iii) ${\cal F}(\phi)=\textsf{U} \ \Leftrightarrow {\cal F}(\neg\phi)=\textsf{U}$.
\end{Lem}

{\em Proof.} (i) By Def. \ref{D:truthvalues} we have immediately:
$${\cal F}(\phi)=\textsf{F}\Leftrightarrow {\cal F}_{i,j}\Vdash \neg\phi\Leftrightarrow {\cal F}(\neg\phi)=\textsf{T}.$$

(ii) By Def. \ref{D:truthvalues} and Def. \ref{D:support} (g) we have:
$${\cal F}(\phi)=\textsf{T}\Leftrightarrow {\cal F}_{i,j}\Vdash \phi\Leftrightarrow  {\cal F}_{i,j}\Vdash \neg\neg\phi\Leftrightarrow {\cal F}(\neg\phi)=\textsf{F}.$$

(iii) Let ${\cal F}(\phi)=\textsf{U}$, i.e., ${\cal F},i,t\not\Vdash \phi$ and ${\cal F},i,t\not\Vdash \neg\phi$. By Def. \ref{D:support} (g), this is equivalently written ${\cal F},i,t\not\Vdash \neg\phi$ and ${\cal F},i,t\not\Vdash \neg\neg\phi$, which means by definition ${\cal F}(\neg\phi)=\textsf{U}$. \telos

\vskip 0.2in

Now if ${\cal F}(\phi), {\cal F}(\psi)\in \{\textsf{T}, \textsf{F}\}$,  the truth values  of $\phi\wedge\psi$ and $\phi\vee\psi$ are computed exactly as in classical logic, as follows from Def \ref{D:support}. So we have to check the  values of $\phi\wedge\psi$ and $\phi\vee\psi$ only in the cases  where some of the  ${\cal F}(\phi)$ and  ${\cal F}(\psi)$ is  $\textsf{U}$.

\begin{Lem} \label{L:wedge}
For every $\phi, \psi$:

(i) If ${\cal F}(\phi)=\textsf{T}$ and ${\cal F}(\psi)=\textsf{U}$, then (a) ${\cal F}(\phi\wedge \psi)=\textsf{U}$ and (b) ${\cal F}(\phi\vee \psi)=\textsf{T}$.

(ii) If ${\cal F}(\phi)=\textsf{F}$ and ${\cal F}(\psi)=\textsf{U}$, then (a) ${\cal F}(\phi\wedge \psi)=\textsf{F}$ and (b) ${\cal F}(\phi\vee \psi)=\textsf{U}$.

(iii) If ${\cal F}(\phi)=\textsf{U}$ and ${\cal F}(\psi)=\textsf{U}$, then (a) ${\cal F}(\phi\wedge \psi)=\textsf{U}$ and (b) ${\cal F}(\phi\vee \psi)=\textsf{U}$.
\end{Lem}

{\em Proof.} (i) (a) Given the assumptions, we have to show that   ${\cal F},i,t\not \Vdash \phi\wedge \psi$ and  ${\cal F},i,t\not \Vdash \neg(\phi\wedge \psi)$, or ${\cal F},i,t\not \Vdash \neg\phi\vee \neg\psi$. Indeed, if ${\cal F},i,t\Vdash \phi\wedge \psi$, then ${\cal F},i,t\Vdash \psi$, which is false since by assumption ${\cal F},i,t\not\Vdash \psi$ and ${\cal F},i,t\not\Vdash \neg\psi$. If on the other hand ${\cal F},i,t. \Vdash \neg\phi\vee \neg\psi$, then ${\cal F},i,t\Vdash \neg\phi$ or ${\cal F},i,t\Vdash  \neg\psi$ which are both false by our assumptions. (b) Since ${\cal F}(\phi)=\textsf{T}$, i.e., ${\cal F},i,t\Vdash \phi$, obviously ${\cal F},i,t\Vdash \phi\vee\psi$ too, and hence ${\cal F}(\phi\vee \psi)=\textsf{T}$.

(ii) (a) Since ${\cal F}(\phi)=\textsf{F}$ we have ${\cal F},i,t\Vdash \neg \phi$, hence ${\cal F},i,t\Vdash \neg \phi\vee \neg\psi$, or ${\cal F},i,t\Vdash \neg(\phi\wedge\psi)$. Thus ${\cal F}(\neg(\phi\wedge\psi))=\textsf{T}$ and hence, by Lemma \ref{L:neg}, ${\cal F}(\phi\wedge\psi)=\textsf{F}$. (b) We show that ${\cal F},i,t\not\Vdash \phi\vee\psi$ and ${\cal F},i,t\not\Vdash \neg\phi\wedge\neg\psi$. Indeed, if ${\cal F},i,t\Vdash \phi\vee\psi$ then ${\cal F},i,t\Vdash \phi$ or ${\cal F},i,t\Vdash \psi$, both of which contradict our assumptions. If also ${\cal F},i,t\Vdash \neg\phi\wedge\neg\psi$, then ${\cal F},i,t\Vdash \neg\psi$, contrary again to our assumption about $\psi$.

(iii) Easy to check. \telos.

\vskip 0.2in

The truth values for $\neg\phi$, $\phi\wedge\psi$ and $\phi\vee\psi$ as follow from Lemmas \ref{L:neg} and \ref{L:wedge} are summarized in the following truth tables:

\vskip 0.2in

\begin{center}
\begin{tabular}{c|c}
$\neg$ & \\
\hline
\textsf{T} & \textsf{F} \\
\textsf{U} & \textsf{U}\\
\textsf{F} & \textsf{T}\\
\end{tabular}
\quad\quad
\begin{tabular}{c|ccc}
$\wedge$ & \textsf{T} & \textsf{U} & \textsf{F}\\
\hline
\textsf{T} & \textsf{T} & \textsf{U} & \textsf{F}\\
\textsf{U} & \textsf{U} & \textsf{U} & \textsf{F}\\
\textsf{F} & \textsf{F} &  \textsf{F}& \textsf{F}\\
\end{tabular}
\quad \quad
\begin{tabular}{c|ccc}
$\vee$ & \textsf{T} & \textsf{U} & \textsf{F}\\
\hline
\textsf{T} & \textsf{T} & \textsf{T} & \textsf{T}\\
\textsf{U} & \textsf{T} & \textsf{U} & \textsf{U}\\
\textsf{F} & \textsf{T} &  \textsf{U}& \textsf{F}\\
\end{tabular}
\end{center}

These truth tables coincide with  those of strong Kleene logic (see \cite{Pa-SJ24} or \cite{Kl52}).

\vskip 0.2in

\textbf{Acknowledgements} I am very indebted to a reviewer of a previous version of the paper who had observed that FOS can be treated as a 3-valued logic and had suspected that  this logic might have similarities with strong Kleene 3-valued logic. It was by following his/her recommendation that I came to prove that the logic of FOS is indeed strong Kleene logic.


\begin{thebibliography}{99}
\bibitem{De18}
    W. Dean, Strict finitism, feasibility, and the sorites, {\em Rev.  Symb. Logic} {\bf 11} (2018), no. 2, 295-346.
\bibitem{Fi75}
    K. Fine, Vagueness, truth and logic, {\em Synth\'{e}se} {\bf 30}
    (1975),  no. 3/4, 265-300.
\bibitem{Itz21}
   Yair Itzhaki, Qualitative versus quantitative representation: a non-standard analysis of the sorites paradox, {\em Linguistics and Philosophy} {\bf 44} (2021), no. 5, 1013-1044.
\bibitem{Ke00}
   R. Keefe, {\em Theories of vagueness,} Cambridge U.P., 2000.
\bibitem{Kl52}
   S. Kleene, {\em Introduction to metamathematics}, North-Holland,  Amsterdam,
   1952.
\bibitem{Pa10}
   P. Pagin, Vagueness and central gaps,  in: {\em Cuts and Clouds,}  R. Dietz and S. Moruzzi (eds.), Oxford University Press 2010, pp. 254--272.
\bibitem{Pa11}
   P. Pagin, Vagueness and domain restriction, in: {\em Vagueness and Language Use,} P. Egr\'{e} and N. Klinedinst (eds.),  Palgrave MacMillan 2011,  pp. 283--307.
\bibitem{Pa-SJ24}
   F. Paoli and G. St. John, Editorial Introduction to the Special Issue on Strong and Weak Kleene Logics, {\em Studia Logica} \textbf{112} (2024), no. 6, 1201-1214.
\bibitem{SN96}
    A.F. Sanders and O. Neumann (Eds), {\em Handbook of
    perception and action, Vol 3: Attention,} Academic Press 1996, pp.
    277-331.
\bibitem{Sh06}
  S. Shapiro, {\em Vagueness in Context,} Clarendon Press -
  Oxford, 2006.
\bibitem{SC82}
    R.N. Shepard and L.A. Cooper, {\em Mental images and their transformations,}
    Cambridge, MIT Press/ Bradford Books, 1982.
    277-331.
\bibitem{So94}
  R.L. Solso, {\em Cognition and visual arts,} Bradford Books,  MIT Press, 1994.
\bibitem{SW95}
  G. Sperling and E. Wechselgartner, Episodic theory of the
  dynamics of spatial attention, {\em Psychological Review} {\bf
  102} (1995), 503-532.
\bibitem{Tz98}
  A. Tzouvaras, Modeling  vagueness by nonstandardness, {\em Fuzzy
  Sets and Systems} {\bf 94} (1998), 385-396.
\bibitem{Tz03}
   A. Tzouvaras, An axiomatization of ``very'' within systems of set theory, {\em Studia Logica} {\bf 73} (2003), 413-430.
\bibitem{Vo79}
  P. Vop\v{e}nka, {\em Mathematics in the Alternative Set
  Theory,}  Teubner Texte, Leipzig, 1979.
\end{thebibliography}
\end{document}